\documentclass[reqno]{amsart}
\usepackage[utf8]{inputenc}
\usepackage{enumitem}

\usepackage{amsmath,amssymb,amsbsy,amsfonts,amsthm,latexsym,amsopn,amstext,mathtools,leftidx,
            amsxtra,euscript,amscd,verbatim,epsf,colordvi,graphics}
\usepackage{ytableau}
\usepackage{color}
\usepackage{pgf,tikz}
\usetikzlibrary{decorations.pathreplacing}
\usetikzlibrary{arrows}
\usepackage{mathrsfs}
\usepackage{graphicx}

\usepackage{hyperref}

\usepackage{cleveref}
\usepackage[margin=11pt,font=footnotesize,labelfont=bf]{caption}
\newtheoremstyle{tplain}{3pt}{3pt}{\rmfamily}{}{\bfseries}{.}{0.5em}{}
\theoremstyle{tplain}

\usepackage{pgf,tikz}
\usepackage{tcolorbox}
\usepackage[enableskew]{youngtab}
\definecolor{darkgreen}{cmyk}{1,0,1,0}
\newtheorem{thm}{Theorem}
\newtheorem{lem}{Lemma}
\newtheorem{ex}{Example}
\newtheorem{cor}{Corollary}
\newtheorem{prop}{Proposition}
\newtheorem{obs}{Remark}
\newtheorem{defi}{Definition}
\def \cal{\mathcal}
\def \rm {\mathrm}
\def \mbf {\mathbf}

\def\rev{{\rm rev}}
\def\Z { \mathbb{Z}}
\def\C { \mathbb{C}}
\def\LRAII {\mathsf{LR}}
\def \r {\mathbf{r}}
\def \t {\mathbf{t}}
\def \redu {\mathrm{red}}
\def\c { \mathrm{c}}
\def \k {\mathfrak{k}}
\def\wt { \mathrm{wt}}

\newcommand{\oo}{\color{blue}}
\newcommand{\blue}{\color{blue}}
\newcommand{\red}{\color{red}}
\newcommand{\green}{\color{green}}

\makeatletter
\newcommand*\bigcdot{\mathpalette\bigcdot@{.5}}
\newcommand*\bigcdot@[2]{\mathbin{\vcenter{\hbox{\scalebox{#2}{$\m@th#1\bullet$}}}}}
\makeatother

\newcommand{\YT}[3]{
\vcenter{\hbox{
\begin{tikzpicture}[x={(0in,-#1)},y={(#1,0in)}] 
\foreach \rowi [count=\i] in {#3} {
 \foreach \e [count=\j] in \rowi {
  \draw (\i,\j) rectangle +(-1,-1);
  \draw (\i-0.5,\j-0.5) node {$#2\e$};
 }
}
\end{tikzpicture}
}}
}

\title[$\mathfrak{k}$- pairs] {
The slack data of the recording tableaux  in the quantum Littlewood-Richardson map determines its inverse:\\  some applications }

\author{Olga Azenhas}
\address{ University of Coimbra, CMUC, Department of Mathematics, Portugal}
\email{oazenhas@mat.uc.pt}



\begin{document}

\begin{abstract}
{ We introduce the slack of a recording tableau in the quantum Littlewood-Richardson (LR) map and show that it inherits the needed data from LR-Sundaram tableaux to define the inverse of the quantum LR map. Notably this enriched slack information packs the suitable reverse Schensted column insertion routes to compute the inverse. The slack  data  is then applied to $\mathfrak{k}$-highest weight symplectic tableaux.}
\end{abstract}
\maketitle

\tableofcontents

\section{Introduction}
This note is a follow-up of the author's paper \cite{azreco} on the recording tableaux of the quantum LR map by Watanabe \cite{watanabe} and its role (by analogy of the G. Thomas \cite{tho78} LR bijection) on the computation of  $\mathfrak{k}$-highest or lowest weight tableaux in a recent
 proof of the Naito–Sagaki conjecture via the branching rule for $\imath$quantum groups \cite{nsw}.

 Our goal with this manuscript  is to surface in a conceptual way the dual information hidden on the recording tableaux in the quantum LR map to exhibit the inverse of the quantum LR map. We look at  the recording tableaux through its decomposition into vertical strips (uniquely defined by a quantum LR recording tableau )  and read  the gaps in the disconnected vertical strips as dual information to define the inverse of the quantum LR map. It turns out that disconnected vertical strips are the interesting ones and we  call to their gaps, \emph{slacks}. We then organize that slack data through slack numbers by counting the gaps, slack vectors defined by the row gap coordinates,  and  associate incidence   slack vectors (see Subsection \ref{subsec:1verticalstrip} and  Definition \ref{def:tei}). Then we show that these tools  inherit the needed information of LR Sundaram tableaux via the quantum LR recording tableaux to guarantee the good reverse bumping routes to compute the inverse of the quantum LR map as proposed in \cite{azreco}.

  We work this in Section \ref{sec:inverse}. Our main results are Proposition \ref{prop:slack}, Theorem \ref{cor:generalhole1} and Corollary \ref{lem:H-L}. In Section \ref{sec:kk}, Example \ref{ex:n=3}, we exhibit several patterns for $\k$-height weight tableaux when $n=3$. The exploration of the case $n=3$ reveals interesting properties on the generation of $\k$-highest weights that should be further studied.

  Section \ref{prel} collects previous information in \cite{watanabe} and \cite{azreco} needed for this follow-up. This manuscript is updated according to the results in the recent preprint by the author  \cite{azreduction}.

\section*{Acknowledgements}
The author acknowledges financial support by the Centre for Mathematics of the University of Coimbra (CMUC, https://doi.org/10.54499/UID/00324/2025) under the Portuguese Foundation for Science and Technology (FCT), Grants UID/00324/2025 and UID/PRR/00324/2025.

\section{Preliminaries}\label{prel}
We follow the notation in \cite{azreco} and \cite{watanabe} and in this section we collect needed information from those sources.
 A \emph{partition} $\gamma$ is a weakly decreasing sequence of nonnegative integers $\gamma_1\ge \gamma_2 \ge \cdots$
such that $\gamma_k =0$ for some $k \ge 1$. The   maximal $i$ such that $\gamma_i> 0$ is called the \emph{number}
of \emph{parts} or \emph{length} of $\gamma$, denoted $\ell(\gamma)$. For each $m\ge 0$, the set of partitions of length at most $m$ is denoted by \emph{$Par_{\le m}$}. We assume the inclusion $Par_{\le m}\subseteq Par_{\le k}$ whenever $k\ge m$. Thus we often write the partition $\gamma$ as a vector
 $\gamma = (\gamma_1, \gamma_2, \dots,\gamma_k)$ for $k \ge  \ell(\gamma)$. The empty partition is the empty sequence $()$ and is regarded as the unique partition of length zero.  The partition $\gamma$ is said to be \emph{even} if $\gamma_{2i-1}=\gamma_{2i}$ for all $i\ge 1$. In other words, all columns of $\gamma$ have even length and necessarily the length of $\gamma$ is even. Given $m\in\Z_{\ge 0}$, $\varpi_m$ denotes the partition of length $m$ whose parts are all $1$, that is, $\varpi_m =(\underbrace{1,\dots,1}_{m})=:(1^m)$.

 A partition $\gamma$ is identified
with its \emph{Young diagram} $D(\gamma)$ which is a left and top justified collection of boxes (or cells)
with $\gamma_k$ many boxes in the $k$th row for all $k \in \mathbb{Z}_> 0$. In particular, the empty Young diagram and  the  partition $()$ are identified. The number of cells of $D(\gamma)$ is  the sum of the parts of $\gamma$ and is denoted by $|\gamma|$.
The boxes or cells of the Young diagram of $\gamma$ are identified by its  coordinates $(i,j)$ in the matrix style, that is, $1\le i\le \ell(\gamma)$ and $1\le j\le \gamma_i$.

Let $\gamma$, $\mu$  partitions   with $\mu \subseteq \gamma$, that is,  $\mu_i\le  \gamma_i$ for all $ i\in\mathbb{Z}_>0 $, or the Young diagram of $\mu$ is a subset of
the Young diagram of $\gamma$. A \emph{
tableau} $T$ of (skew) shape $\gamma/\mu$ is a map (or a filling of $D(\gamma)$)
 $$T:D(\gamma)\rightarrow \mathbb{Z}_\ge0, \; (i,j)\mapsto T(i,j),$$
assigning a positive integer to each box of $\gamma\setminus\mu$ and   $0$  to the boxes corresponding to $\mu$. Denote by $\rm{Tab}(\gamma/\mu)$ the set of tableaux of shape $\gamma/\mu$.  We say that the tableau $T$ is \emph{semi-standard} if an addition the assignment is
such that it is weakly increasing as we go from left to right along a row and strictly
increasing as we go from top to bottom along a column excluding the boxes in $\mu$,
\begin{align}&T(i,j)\le T(i,j+1),\; T(i,j)<T(i+1,j), \mbox{ for all $(i,j)\in D(\gamma)\setminus D(\mu)$,}\nonumber\\
&\mbox{ and }  T(i,j)=0, \mbox{ for $(i,j)\in  D(\mu)$,}
\nonumber
\end{align}
where we set $T(a,b):=\infty$ if $(a, b) \notin D(\gamma)$.
 Usually $T(i,j)$ is just referred as
the entry in the box $(i,j)$ and we omit the zeroes  in the boxes of $\mu$. A positive integer
$m\ge \ell(\gamma)$ will be fixed and $[0, m]:=\{0, 1, \dots, m\}$ will be used as a co-domain for the map $T$. We call
$[m]:=\{1, \dots, m\}$ the alphabet of the semi-standard tableau $T$.
 In this case, we will denote the \emph{set} of \emph{semi-standard} \emph{tableaux} of \emph{shape} $\gamma/\mu$ by
$SST_k(\gamma/\mu)$. When $\mu=()$, we just write $SST_m(\gamma)$. The \emph{weight} or \emph{content} of $T$ is the nonnegative vector $\mathrm{wt}(T)=(T[1],\dots,  T[m])$, where $T[i]:=\#\{(a,b)\in D(\gamma):T(a,b)=i\}$ for $i\in[m]$, that is, $T[i]$ is  the number of occurrences of $i$ in the tableau $T$.
The \emph{reverse} \emph{row} \emph{word} of a semi-standard tableau
$T$, denoted $w(T)=w_1\cdots w_l$, with $l$ the number of non zero entries of $T$,  is obtained by reading the entries of its rows (excluding the entry 0) right to left
starting from the top row and proceeding downwards. The \emph{column} \emph{word} of a semi-standard tableau $T$ denoted $w^{col}(T)=w'_1\cdots w_l'$  is the sequence of entries obtained by  reading the entries (excluding the entry $0$) of its columns
 from bottom to top and left to right.
The \emph{reverse} \emph{column} \emph{word}  of $T$, $w'_l\cdots w'_1$.
is obtained by reading the entries of its columns (excluding the entry 0) right to left
starting from the top and proceeding downwards.
The \emph{weight} of the \emph{word}  $w(T)$ is the weight of $T$.

 A \emph{Yamanouchi} \emph{word} is a word $u_1 \cdots u_l$ such
that, for each $1 \le k \le l$,  the  weight of the subword  $u_1 \cdots u_k$ is a partition.
 A Yamanouchi tableau of shape $\mu$ is a semi-standard tableau of shape $\mu$ whose word is Yamanouchi, equivalently $Y(i,j)=i$ for all $1\le i\le \ell(\mu)$ and $1\le j\le\mu_i$. A Yamanouchi tableau of shape $\mu$ is the unique semi-standard tableau whose weight is $\mu$.

For $\lambda\subseteq\gamma$ partitions, we write $\lambda\subset_{vert} \gamma$ to mean that  $\gamma/\lambda$  is a \emph{vertical} \emph{strip}, that is, the skew-diagram $\gamma/delta$  at most one cell in each row. The \emph{size} of $\gamma/\lambda$ is $|\gamma/\lambda|:=|\gamma|-|\lambda|$ the number of its cells.

\subsection{LR-Sundaram tableaux and orthogonal  transposition symmetry} Let $n\in\mathbb{N}$. The set  of semistandard Young tableaux  of shape $\lambda$ with entries in $[1, 2n] :=
\{1,\dots 2n\}$ is denoted by $SST_{2n}(\lambda)$ and necessarily $\lambda$ has at most $2n$ parts, $\lambda\in Par_{\le 2n}$.

Let $\mu\subseteq \lambda \in Par_{\le 2n}$. A tableau $T\in SST_{2n}(\lambda/\mu)$ is said to be a Littlewood-Richardson (LR) tableau \cite{LR} if its reverse column word
is a Yamanouchi word. This subset of $SST_{2n}(\lambda/\mu)$ consisting of LR tableaux is denoted by $LR_{2n}(\lambda/\mu)$.
A partition $\nu$ is said to be even if  all columns of $\nu$ have even length and necessarily the length of $\nu$ is even. In this case, $\nu^t$, the transpose or conjugate of $\nu$, has even parts.
A Littlewood-Richardson-Sundaram tableau or symplectic Littlewood-Richardson tableau \cite{sundaram} is a certain LR tableau as follows.
\begin{defi}\label{def:lrs} Let $\lambda \in Par_{\le 2n}$ and $\mu \in Par_{\le n}$. A semi-standard tableau $T \in SST_{2n}(\lambda/\mu)$  is said to be
an $n$-symplectic Littlewood-Richardson tableau or Littlewood-Richardson-Sundaram tableau if it satisfies the following:
\begin{enumerate}

\item [(1)]
The reverse column word of $T$
is a Yamanouchi word.

\item [$(2)$] The sequence $\rm{wt}(T) = (T[1], T[2],\dots, T[2n])$ is an even partition,

\item [$(3)$] $T(n + i, 1) \ge 2i$ for every $i\ge 0$.

\end{enumerate}
\end{defi}

For $\lambda \in Par_{\le 2n}$ and $\mu \in Par_{\le n}$, the subset of $SST_{2n}(\lambda/\mu)$ satisfying conditions $(1)$, $(2)$ and  $(3)$ in Definition \ref{def:lrs} is denoted by $LRS_{2n}(\lambda/\mu)\subseteq LR_{2n}(\lambda/\mu)$.

\begin{obs} Let $T\in LRS_{2n}(\lambda/\mu)$ of content $\nu$. Let $N:=\nu_1$ and $\nu^t=(\nu_1^t,\dots,\nu^t_{N})$ the transpose of $\nu$. Put  $\mu^{(0)}:=\lambda$ and $\mu^{(N)}:=\mu$. For $k=1,\dots,N$,
 define $\mu^{(k)}$ the shape obtained  by erasing from NE to SW in $T_{|\mu^{(k-1)}}$, $T$ restricted to the shape  $\mu^{(k-1)}$, the first rightmost $\nu^t_k$ cells filled with $1,2,\dots, \nu^t_k$. Hence $$\mu^{(0)}=\lambda\supset_{vert} \mu^{(1)}\supset_{vert} \cdots\supset_{vert}\mu^{(N)} =\mu$$
This decomposes the shape $\lambda/\mu$ into $N$ vertical strips
$$J_1:=\mu^{(0)}/\mu^{(1)},J_2:= \mu^{(1)}/\mu^{(2)}, J_i=\mu^{(i-1)}/\mu^i,\dots, J_{N}=\mu^{(N-1)}/\mu^{(N)}$$

 Then, for $k=1,\dots,N$,
the vertical strip $J_k=\mu^{(k-1)}/\mu^{(k)}$ has  $\nu^t_k=|\mu^{(k-1)}|-|\mu^{(k)}|$ cells and  is filled in  $T$, top to bottom,  with
$12\cdots \nu_k^t$. 

Concatenating those column words as $12\cdots \nu^t_{N}\bigcdot\cdots\bigcdot 12\bigcdot\cdots\bigcdot\nu^t_2\bigcdot 12\cdots \nu^t_{1}$ gives the reverse-column word of the  Yamanouchi tableau of shape $\nu$ which is also obtained by rectification of the word of $T$.

\end{obs}

Let $T\in LRS_{2n}(\lambda/\mu)$ of weight $\nu$ and $J_1,\dots,J_{N}$, with $N=\nu_1$, its decomposition into vertical strips each strip filled in $1,2,\dots, \nu^t_i\in 2\mathbb{Z}$ for $i=1,\dots,N$.
Relabel each vertical strip $J_i=\mu^{(i-1)}/\mu^{(i)}$ with $\nu^t_i$ $i$'s, for $i=1,\dots,N$. We get a new tableau $Q\in Tab(\lambda/\mu)$ 
of weight $\nu^t$. The  map on $LRS_{2n}(\lambda/\mu)$ returning such $Q$ is denoted by $\lozenge$ and is illustrated below.
\begin{ex} \label{ex:symmetry}Let $n=3$, $\lambda=(4,3,2,2,1,0)$, $\lambda^\vee=(3,2,2,1,0)$,   $\mu=(3,1,0)$, $\mu^\vee=(4,4,4,3,1)$, $\mu^t=(2,1,1,0)$, $\nu=(3,3,1,1,0^2)$, $\nu^t=(4,2,2)$ with even parts and $T\in LRS_6(\lambda/\mu,\nu)$
\begin{align} & T=\YT{0.13in}{}{
 {{},{},{},{1}},
 {{},{1},{2},{}},
 {{1},{2},{},{}},
 {{2},{3},{},{}},
 {{4},{},{},{}},
}\in LRS_6(\lambda/\mu,\nu)\overset{ \sim}{\underset{ \lozenge}\rightarrow} Q=\YT{0.13in}{}{
 {{},{},{},{1}},
 {{},{2},{1},{}},
 {{3},{2},{},{}},
 {{3},{1},{},{}},
 {{1},{},{},{}},
}=\lozenge(T)\in Tab(\lambda/\mu,\nu^t)\overset{ \sim}{\underset{\pi\circ \scriptstyle tr}\rightarrow}
Q^{\pi\circ tr}=\YT{0.13in}{}{
 {{},{},{},{},{1}},
 {{},{},{},{1},{}},
 {{},{1},{2},2,{}},
 {{1},{3},{3},{},{}},
}
=\blacklozenge T\nonumber\\
&\blacklozenge T \in LR({\mu^\vee}^t/{\lambda^\vee}^t,\nu^t).\nonumber
\end{align}

The bijection $\blacklozenge=\pi\circ \scriptstyle tr\circ \lozenge$ is the LR orthogonal symmetry map \cite{acm09}.

The tableaux $Q$ and $T$  are both defined by the sequence of nested partitions $$\mu^{(0)}=\lambda\supset_{vert} \mu^{(1)}=(3,2,2,1,0^2)\supset_{vert} \mu^{(2)}=(3,1,1,0^3)\supset_{vert}\mu^{(3)} =\mu$$
which  decomposes $\lambda/\mu$ into three vertical strips $$J_1:=\mu^{(0)}/\mu^{(1)},J_2:= \mu^{(1)}/\mu^{(2)}, J_3=\mu^{(2)}/\mu^{(3)}$$
\end{ex}
The bijection $\lozenge$  is so natural that  we can intertwine LR-Sundaram tableaux  with their $Q$ presentations.
We define the set
\begin{align}\label{rec:set}Rec_{2n}(\lambda/\mu):=\lozenge LRS(\lambda/\mu)\\
LRS_{2n}(\lambda/\mu)\overset{\sim}\longrightarrow Rec_{2n}(\lambda/\mu).
\end{align}

The set $Rec_{2n}(\lambda/\mu)$ is characterized in \cite{watanabe} (and combinatorially in \cite{azreco}) by translating the conditions in Definition \ref{def:lrs} to $Q$.
\begin{prop}\cite{watanabe}\label{prop:tilderec} Let $\lambda \in Par_{\le 2n}$ and $\mu \in Par_{\le n }$ be such that $\mu\subseteq\lambda$. The  set $Rec_{2n}(\lambda/\mu)\subseteq Tab(\lambda/\mu)$ satisfy  the following  conditions:
\begin{enumerate}
\item[$(R1)$] The entries of $Q$ strictly decrease along the rows from left to right.

\item[$(R2)$] The entries of $ Q$ weakly decrease along the columns from top to bottom.

\item[$(R3)$] For each $k > 0$, the number $Q[k]\in 2\mathbb{Z}$.

\item[$(R4)$] For each $k > 0$, it holds that

$$Q[k] \ge 2(\ell(\mu^{(k-1)})- n),$$
\noindent where $\mu^{(k-1)}$ is the partition such that
\begin{align}D(\mu^{(k-1)}) = D(\mu) \cup\{(i, j) \in D(\lambda/\mu) | Q(i, j) \ge  k\}.\label{nested}\end{align}

\item[$(R5)$] For each $r, k >0 $, let $Q_{\le r}[k]$ denote the number of occurrences of $k$ in $Q$ in the
$r$-th row or above. Then, the following inequality holds:
$$Q_{\le r}[k + 1] \le Q_{\le r}[k].$$

\end{enumerate}

\end{prop}

\medskip
\begin{obs} Let $N$ be the largest entry in $Q$ where we are assuming the blank boxes filled with $0$. When $N=0$ one has $\lambda=\mu \Leftrightarrow\lambda/\lambda=()\Leftrightarrow \rm{wt}(Q)=()$. Let $N\ge 1$.
Since $Q[k]\in 2\mathbb{Z}$ and  $Q[k+1]\le Q[k]\le \ell(\lambda)\le 2n$, for any $k>0$,
 the weight of $Q$ is   the partition $\rm{wt}(Q)=(Q[1], \dots, Q[N])$
 and  its  transpose or conjugate is an even partition $\nu=(\nu_1=N,\dots,\nu_{Q[1]})$.
Thus
 $Q$ is also defined by the sequence of nested partitions in \eqref{nested}

\begin{align}\mu^{(0)}=\lambda\supset_{vert} \mu^{(1)}\supset_{vert} \cdots\supset_{vert}\mu^{(N)} =\mu,~~
\end{align}

\noindent  where $\mu^{(i-1)}/\mu^{(i)}$ is a vertical strip of even length $Q[i]\ge 2$ satisfying $(R4)$ and $(R5)$ conditions, for $i=1,\dots,N$.

\end{obs}

\subsection{The quantum LR bijection}

A semistandard tableau $S \in SST_{2n}(\mu)$ is said to
be symplectic if
$S(k, 1) \ge 2k-1$,   for all $  k \in  [ 1,\ell(\lambda)]$.
Let $SpT_{2n}(\mu)\subseteq SST_{2n}(\mu)$ denote the set of symplectic semi-standard tableaux of shape $\mu$, with entries in $[1,2n]$ \cite{king76,watanabe}. The definition of symplectic tableau forces that $\mu$ has at most $n $ parts \cite[Proposition 2]{azreco}, and in this case $\mu\in Par_{\le n}$.


Let $\lambda\in Par_{\le 2n}$. The quantum Littlewood–Richardson bijection of type AII, $LR^{AII}$, is the
algorithm  \cite[Section 3.1]{watanabe,watanabe25} which takes $T \in  SST_{2n}(\lambda)$ as input, and returns a pair of tableaux $(P^{AII}(T),Q^{AII}(T))\in SpT_{2n}(\mu)\times Rec_{2n}(\lambda/\mu)$ for some $\mu\subseteq \lambda$,
\begin{align}\label{wat0}LR^{AII}:& SST_{2n}(\lambda) \overset{\sim}\rightarrow\bigsqcup_{\begin{smallmatrix}\mu\in Par_{\le n}\\
\mu\subseteq\lambda\end{smallmatrix}}SpT_{2n}(\mu)\times Rec_{2n}(\lambda/\mu),\nonumber\\
&\qquad\qquad T\mapsto (P^{AII}(T),Q^{AII}(T))
\end{align}
where $Q^{AII}(T)$ is called a recording tableau of skew shape $\lambda/\mu$. In particular, if $\lambda$ has at most $ n$ parts then $SpT_{2n}(\lambda)\subseteq SST_{2n}(\lambda)$ and for $T$  symplectic, $Rec_{2n}(\lambda/\lambda)=\{()\}$,  $LR^{AII}(T)=(P^{AII}(T),Q^{AII}(T))=(T,())$.

Then
\begin{align}\label{rec=rec2}Rec_{2n}(\lambda/\mu) = \{Q^{AII}(T) | T \in SST_{2n}(\lambda) \mbox{ such that $sh(P^{AII}(T)) = \mu $}\}
\end{align}
and is called the set of recording tableaux of shape $\lambda/\mu$ of $LR^{AII}$.

\section{The algorithm computing the inverse of the quantum LR map}\label{sec:inverse}

\subsection{Preliminaries on removals}
We fix $l \in [0, 2n]$ and $\mathbf{a} = (a_1, \dots , a_l)$ a column  in $SST_{2n}(\varpi_l)$.
We often regard $\mathbf{a}$ as a set. We recall the following useful property on removals \cite{watanabe} in the computation of the inverse reduction. For each $x\in\mathbb{Z}$, set $s(x)=x+1$, if $x\notin 2\mathbb{Z}$, and $s(x)=x-1$,  if $x\in 2\mathbb{Z}$.
\begin{prop}\cite[Proposition 4.2.2, 4.2.7, Corollary 4.2.8]{watanabe}\label{prop:rem}
\begin{enumerate}

\item[(1)] $i\in [1,l]$ and $a_j\notin rem(\mbf{a})$, for $j\in[i,l]$ then $rem(\mbf{a})=rem(a_1,\dots,a_{i-1})$.

\item[(2)] for each $i\in [1,l]$, $a_i\in rem(\mbf{a})$ if and only if one of the following holds
\begin{enumerate}
\item[(a)] $a_i$ odd, $i<l$, $a_{i+1}=a_i+1$ and $a_i<2i-|rem(a_1,\dots,a_{i-1})|$

\item[(b)] $a_i$ even, $i>1$, $a_{i}=a_{i-1}+1$ and $a_i<2i-|rem(a_1,\dots,a_{i-2})|-1$
\end{enumerate}

\item [(3)]  $a_i \in rem(\mbf{a})$ if and only if $s(a_i) \in rem(\mbf{a})$.
Consequently, $|rem(\mbf{a})|\in 2\mathbb{Z}$.



\end{enumerate}

\end{prop}

\subsection{The algorithm computing the inverse LR map and the  row index slack vectors}

Let  $\mu\subseteq\lambda\in Par_{\le 2n}$ with  $\mu\in Par_{\le n}$. Fix arbitrarily $S\in SpT_{2n}(\mu)$. If $\lambda=\mu$,  $Rec_{2n}(\lambda/\lambda)=Rec_{2n}(())$ and  ${\LRAII^{AII}}^{-1}(S, ())=S$.
\begin{align}{\LRAII^{AII}}^{-1}:SpT_{2n}(\mu)\times Rec_{2n}(())&\longrightarrow SpT_{2n}(\mu)\\
(S,())&\mapsto S,~\redu(S)=S
\end{align}

Let us consider $\mu\subset\lambda$ and define

\begin{align}\label{Rtilde}
&{LR_{|S}^{AII}}^{-1}:\{S\}\times  Rec_{2n}(\lambda/\mu)\longrightarrow SST_{2n}(\lambda)\nonumber\\
&\qquad\qquad\qquad \quad(S,Q)\mapsto S^{\r^{(N)}\cdots\r^{(2)} \r^{(1)}}=  c \circ (\rm{red}_{t_0^{(1)}}^{-1},\rm{id})\circ (\underset{{\mu^{1}/{\mu^{1}}'}} \leftarrow
\bigg(\cdots  \nonumber\\
&(\underset{{\mu^{N-2}/{\mu^{N-2}}'}}\leftarrow\bigg( c \circ (\rm{red}_{t_0^{(N-1)}}^{-1},\rm{id})\circ (\underset{{\mu^{N-1}/{\mu^{N-1}}'}} \leftarrow
\bigg(c \circ (\rm{red}_{t_0^{(N)}}^{-1},\rm{id})\circ (\underset{{\mu^{N}/{\mu^{N}}'}}\leftarrow S)\bigg))\bigg))\cdots)\bigg))
\end{align}
where $N>0$ is the number of vertical strips of $Q$, $\lambda=\mu^0\supset_{vert}\mu^1\supset_{vert}\cdots \supset_{vert} \mu^N=\mu$, $[\r^{(N)},\cdots,\r^{(2)}, \r^{(1)}]$ is the slack vector sequence of $Q$ that encodes the slackness of this sequence of vertical strips to be defined in the next two susections, $\rm{red}_{t_0^{(i)}}^{-1}$ is the inverse of the reduction map, $\redu$, on $SpT_{2n}(\varpi_{t^{(i)}_0})$, or the expanding map via the reverse removal, and $\c$ is the concatenation in the plactic monoid.  The basic operation
$$\c \circ (\rm{red}_{t_0}^{-1},\rm{id})\circ (\underset{{\mu/\mu'}}\leftarrow S)$$
is defined for $N=1$ in Theorem \ref{thm:verygeneralstrip} in the next subsection.
For the reverse bumping routes properties we refer to \cite{azreco} and \cite{fulton}.

\subsection{ 1-vertical strip recording tableaux, slacks and  slack  vectors }\label{subsec:1verticalstrip}
Let $\mu\subset_{vert}\lambda$  with $\ell(\mu)\le n\le l=\ell(\lambda)$. Let $t_0=\ell(\mu)-l_0$ where $l_0$ is the number of cells in the vertical strip $\lambda/\mu$ with row coordinates in $[1,\ell(\mu)]$. We call to $t_0$ the \emph{slack} of the vertical strip $\lambda/\mu$.

Let  $\mathbf{r}=\{r_1<\dots< r_{t_0}\}\subseteq [1,\ell(\mu)]$ be the complement of the set of the row coordinates of the  cells of the vertical strip $\lambda/\mu$ in $[1,\ell(\mu)]$. We call to $\r$ the \emph{slack} \emph{row} \emph{index} \emph{vector} of the vertical strip $\lambda/\mu$, and define $\delta_\mathbf{r}$ to be  the  \emph{incidence} \emph{vector} of the \emph{subset} $\mathbf{r}$  of $ [1,\ell(\mu)]$, that is, $\delta_\r=(x_s)_{s\in [1,\ell(\mu)]}$, $x_s=1$ if $s\in\r$ and $0$ otherwise. We call to $\delta_\mathbf{r}$ the \emph{slack} \emph{incidence} \emph{vector} of the vertical strip $\lambda/\mu$.

Define $\mu'\subset_{vert} \mu$ such that $\mu'_i=\mu_i-1$ if $i\in\r=\{r_1,\dots, r_{t_0}\}$ and $\mu_i$, otherwise. That is $\mu'=\mu-\delta_\mathbf{r}$. Then the vertical strip $\mu/\mu'$ has $t_0=\ell(\mu)-l_0$ cells with row coordinates   $\r=(r_1<\dots< r_{t_0})$ the  \emph{slack} \emph{row} \emph{index vector} of vertical strip $\lambda/\mu$; and \begin{align}\lambda=\mu-\delta_\mathbf{r}+\varpi_{\ell(\lambda)}=\mu'+\varpi_{\ell(\lambda)}\Leftrightarrow \varpi_{\ell(\lambda)}=\lambda-\mu'
\end{align}

Let us write $\lambda=(\lambda_0,\varpi_{l-\ell(\mu)})$ where
 $\mu\subseteq_{vert}\lambda^0\in Par_{\le n}$ such that $\ell(\lambda^0)=\ell(\mu)$ and, thus,  $\lambda^0/\mu$ is a vertical strip with $l_0=|\lambda_0|-|\mu|\le \ell(\mu)$ cells, and slack $t_0=\ell(\mu)-l_0$.

  Let $Q\in
Rec_{2n}((\lambda_0,\varpi_{l-\ell(\mu)})/\mu)$. Then $Q[1]+t_0=l$. Furthermore,
the  conditions $(R3)$, $(R4)$, in  Proposition \ref{prop:tilderec},  impose relations between $l=\ell(\lambda)$ and the  slack $t_0$ of $\lambda/\mu$ (or $\lambda^0/\mu$),
\begin{align}\label{conditions}&l\le 2n-t_0\Leftrightarrow l+t_0\le 2n\Leftrightarrow t_0\le 2n-l,\\
&~\mbox{ and }\nonumber\\
&~0\le l-t_0 \in 2\mathbb{Z}.
\end{align}

We  define the reverse column insertion  data of $Q$ to be the vertical strip $\mu/\mu'$. For simplicity, we often abuse notation and identify the \emph{slack row} \emph{index} \emph{vector} $\r$ with $\mu/\mu'$ and use the slack row index vector $\r$ to apply reverse column insertion to an $S\in SpT_{2n}(\mu)$. We also often  say the \emph{slack} $t_0$, the \emph{slack row} \emph{index} \emph{vector} $\r$ and the \emph{slack} \emph{incidence} \emph{vector} $\delta_\r$ of $Q$ in the sense that the vertical strip $\lambda/\mu$ defines $Q$.

Consequently, the $1$-vertical strip recording tableau of shape $\lambda/\mu$ with slack $t_0$ and $l=\ell(\lambda)$ is  characterized as
\begin{align} \label{rec:1vertical} Rec_{2n}((\lambda_0,\varpi_{l-\ell(\mu)})/\mu)=\left\{Q=\YT{0.15in}{}{
{{},{},{},},
 {{},{},{}1},
 {$\vdots$,{}},
 {{},1},
 {{}},
 {{1}},
 {{1}},
 {{$\vdots$}},
 {{1}},
 {{1}},
},~\ell(\mu)\le l,~~ 0\le l- t_0\in 2\mathbb{Z},~~ l\le 2n-t_0\right\}
\end{align}



\begin{ex}\label{ex:q}
For $n=6$, let
\begin{align*}& Q=\YT{0.15in}{}{
 {{},{},{},{}},
 {{},{},,1},
 {{},{},},
 {{},{},},
 {{},{1}},
 {{}},
 {{1}},
 {{1}},
}\in Rec_{2n}((\lambda_0,\varpi_{l-\ell(\mu)})/\mu),\nonumber
\end{align*}
where $ \mu=(4,3,3,3,1,1)\subset_{vert}\lambda^0=(4,4,3,3,2,1)\subseteq \lambda=(\lambda_0,\varpi_{l-\ell(\mu)})=(\lambda_0,\varpi_{2})$, and $\ell(\lambda^0)=\ell(\mu)=6$, $ l_0= 2$, slack $t_0=\ell (\mu)-l_0=4$, $Q[1]+t_0=l(\lambda)=8\le 12-4$.

One has $\mu'=\mu-\delta_\mathbf{r}=\mu-(1,0,1,1,0,1)=(3,3,2,2,1,0)$ where $\mathbf{r}=\{1,3,4,6\}\subseteq [1,\ell(\mu)]$ is the slack row index vector of $\lambda/\mu$, and
$$\begin{array}{llllll}\mu/\mu'={\tiny\ydiagram{3+1,0,2+1,2+1,0,0+1}}&\qquad \r=\YT{0.15in}{}{
 {1},
 {3},
 {4},
 {6},
}&\qquad  \delta_\r=\YT{0.15in}{}{
 {1},
 {0},
 {1},
 {1},
 {0},
 {1}
}
\end{array}$$
\end{ex}
Indeed $\varpi_{\ell(\lambda)}=\lambda-\mu'$.

Let $S\in SpT_{2n}(\mu)$ and $Q\in Rec_{2n}((\lambda_0,\varpi_{l-\ell(\mu)})/\mu)$.
Let  $S'_\r=(a_1,\dots,a_{t_0})\in Sp_{2n}(\varpi_{t_0})$ be the set of the bumped elements from the first column $S_1$ of $S$ by applying successively the reverse column insertion to the  rows of $S$, as prescribed by slack row index vector $\r=(r_1<\dots<r_{t_0})$, going from the largest to the smallest row;   and let $S'_{\ge 2}$ be  the remained tableau after the application of that reverse column insertion to  $S$.

 For the returned pair $(S'_\r,S'_{\ge 2})$ obtained under the action of the reverse Schensted insertion to the entries in the cells  $\mu/\mu'$ of $S$, put
 \begin{align}\label{revschensted}(S'_\r,S'_{\ge 2})=:(\underset{{\mu/\mu'}}\leftarrow S), \mbox{  or $(S'_\r,S'_{\ge 2})=:(\underset{{\r}}\leftarrow S)$}
 \end{align}
 where $S'_\r\subseteq
  S_1$ is the column comprising the $t_0$ bumped entries from $S_1$, and $S'_{\ge 2}$ is what remains after the bumping applied to the entries in the $t_0$ cells of $\mu/\mu'$ in $S$.

When $l_0=\ell(\mu)\Leftrightarrow\r=()$,  $\mu'=\mu$, and the bumped column $S'_\r=()$ is empty, $((),S):=(\underset{{\mu/\mu}}\leftarrow S)$.

When $l_0=0\Leftrightarrow\r=[1,\ell(\mu)]$, $\mu'=\mu-(1^{\ell(\mu)})$ and the bumped column $S'_\r=S_1$, $(S_1,S_{\ge 2}):=(\underset{{\mu/\mu'}}\leftarrow S)$.

Since $S$ is symplectic and $S'_\r$ is contained in the first column of $S$,  $S'_\r\subseteq SpT_{2n}(\varpi_{t_0})$. (Any subset of a symplectic column is still symplectic.)

For $Q\in Rec_{2n}((\lambda_0,\varpi_{l-\ell(\mu)})/\mu)$ \eqref{rec:1vertical}
\begin{align}\label{tildeR2}&{\LRAII^{AII}}^{-1}(S, Q)= \c \circ (\rm{red}_{t_0}^{-1},\rm{id})\circ (\underset{{\mu/\mu'}}\leftarrow S)\nonumber\\
&=\c\circ (\rm{red}_{t_0}^{-1},\rm{id})(S'_\r,S'_{\ge 2})\nonumber\\
&=\c\circ (\rm{red}_{t_0}^{-1}(S'_\r),S'_{\ge 2})\nonumber\\
&:=S^\r
\end{align}
where $\c$ means concatenation in the plactic monoid and $\rm{red}_{t_0}^{-1}$ means \emph{reverse or the inverse} \emph{reduction map} applied to a  column in $SpT_{2n}(\varpi_{t_0})$, a symplectic column of length the slack number  $t_0$ of $Q$, according to the rules in Theorem 7 in \cite{azreduction}. The reduction map $\redu$ on $SST_{2n}(\varpi_l)$ is injective \cite[Proposition 4.3.6, Corollary 4.4.3]{watanabe} and returns symplectic columns of  length  $0\le k\le \min(l,2n-l)$ and $l-k\in 2\Z$. Hence the inverse $\rm{red}_{t_0}^{-1}$ on $SpT_{2n}(\varpi_{t_0})$ exists

\begin{align} \rm{red}_{t_0}^{-1}: SpT_{2n}(\varpi_{t_0})\rightarrow SST_{2n}(\varpi_{l})
\end{align}
and the rules in \cite{azreduction}
exhibit it. 

Let $S'_\r=(a_1,\dots,a_{t_0})\in SpT_{2n}(\varpi_{t_0})$  be the column of bumped entries  as in \eqref{revschensted}.
Then
\begin{align}\label{inversered}
\rm{red}_{t_0}^{-1}(S'_\r):=T\in SST_{2n}(\varpi_l)
\end{align}
where $T$ is  explicitly defined in  full generality in \cite[Theorem 7 ]{azreduction}, and, in particular, for certain fundamental patterns as in Theorem 2, Theorem 4 and Theorem 6 in \cite{azreduction}.

\begin{thm}\cite{azreco}\label{thm:verygeneralstrip} With the set up above where $S\in SpT_{2n}(\mu)$ and $Q\in  Rec_{2n}((\lambda^0,\varpi_{l-\ell(\mu)})/\mu)$, $l=\ell(\lambda)$, has slack row index vector  $\mathbf{r}=\{r_1<\dots< r_{t_0}\}$, one has the following assertion:

\begin{align}\label{inverse}&{LR^{AII}}^{-1}(S, Q)= c \circ (\rm{red}_{t_0}^{-1},\rm{id})\circ (\underset{{\mu/\mu'}}\leftarrow S)\nonumber\\
&=c\circ (\rm{red}_{t_0}^{-1},\rm{id})(S'_\r,S'_{\ge 2})\nonumber\\
&=c\circ (\rm{red}_{t_0}^{-1}(S'_\r),S'_{\ge 2})\nonumber\\
&=c\circ (T,S'_{\ge 2})\nonumber\\
&=T.S'_{\ge 2}\nonumber\\
&:=S^\r\in SST_{2n}(\lambda)
\end{align}
  such that $T$ is explicitly  given \cite[Theorem 7]{azreduction}, $S'_{\ge 2}\in SpT_{2n}(\mu')$, $\mu'=\mu-\delta_\r$ and  $\lambda=\mu'+\varpi_l$. 
\end{thm}

\begin{obs} \label{nullslack}If  the slack of  $Q$ is  $t_0=0$, then $l\in 2\Z$ and $\rm{red}_{0}^{-1}$ means reverse reduction in $[1,l]$ applied to a column of length $0$. That is, $\rm{red}_{0}^{-1}(())=(12\dots l)$ with $l\in 2\Z$ $\Leftrightarrow \redu(12\dots l)=()\Leftrightarrow \rm{rem}(12\dots l)=(12\dots l)$. 

\begin{align}&{LR^{AII}}^{-1}(S, Q)=\c \circ (\rm{red}_{0}^{-1},\rm{id})\circ (\underset{{\mu/\mu}}\leftarrow S)
=c\circ (\rm{red}_{0}^{-1},\rm{id})((),S)=c\circ (\YT{0.15in}{}{
 {{1}},
 {{2}},
 {{\vdots}},
 {{l}},
 },S)=\YT{0.15in}{}{
 {{1}},
 {{2}},
 {{\vdots}},
 {{l}},
 }\bigcdot S
 \end{align}

\begin{ex}\label{ex:qq} We resume to Example \ref{ex:q}. Let $n=6$, $Q\in Rec_{2n}((\lambda_0,\varpi_{l-\ell(\mu)})/\mu)$ as in Example \ref{ex:q} equipped with $t_0=4$, $\r=(1,3,4,6)$ and $\delta_\r=(1,0,1,1,0,1)$ and $\mu'=\mu-\delta_\r$.
\begin{enumerate}
\item
 Let $U=\YT{0.13in}{}{
 {{2},{2},2,2},
 {{3},{3},3},
 {6,6,6},
 {7,7,7},
 {10},
 {11},
}\in SpT_{12}(\mu)$,

 $(\underset{{\r}}\leftarrow U)=\big((2,6,7,11),
\YT{0.13in}{}{{{2},{2},2},
 {{3},{3},3},
 {6,6},
 {7,7},
 {10},
}\big)
=(\big(2,6,7,11),U'_{\ge 2}\big)$, $U'_{\ge 2}\in SpT_{12}(\mu')$,
and $$ \rm{red}_{4}^{-1}(2,6,7,11)=\YT{0.13in}{}{{{2}},
 {3},
 {4},
 {6},
 {7},
 {9},
 {10},
 {11}
}$$

Then \begin{align}\label{tildeR3}&{\LRAII^{AII}}^{-1}(U, Q)= \c \circ (\rm{red}_{4}^{-1},\rm{id})\circ (\underset{{\r}}\leftarrow U)\nonumber\\
&=\c\circ (\rm{red}_{t_0}^{-1},\rm{id})((2,6,7,11),U'_{\ge 2})\nonumber\\
&=\c\circ (\rm{red}_{t_0}^{-1}(2,6,7,11),U'_{\ge 2})\nonumber\\
&=\YT{0.13in}{}{
{{2}},
 {3},
 {4},
 {6},
 {7},
 {9},
 {10},
 {11},
}\bigcdot\YT{0.13in}{}{{2,{2},2},
 {{3},{3},3},
 {6,6},
 {7,7},
 {10},
}=\YT{0.13in}{}{{2,2,{2},2},
 {{3},3,{3},3},
 {4,6,6},
 {6,7,7},
 {7,10},
 {9},
 {10},
 {11},
}:=U^{\r}  \in SST_{12}(\lambda)
\end{align}

\item
Let $V=\YT{0.13in}{}{
 {{1},{1},1,1},
 {{4},{4},4},
 {5,5,5},
 {8,8,8},
 {9},
 {12},
}\in SpT_{12}(\mu)$, $(\underset{{\r}}\leftarrow V)=\big((1,5,8,12),
\YT{0.13in}{}{{{1},{1},1},
 {{4},{4},4},
 {5,5},
 {8,8},
 {9},
}\big)
=(\big(1,5,8,12),V'_{\ge 2}\big)$ where $V'_{\ge 2}\in SpT_{12}(\mu')$,
and $ \rm{red}_{4}^{-1}(1,5,8,12)=
\YT{0.13in}{}{
{{1}},
 {3},
 {4},
 {5},
 {8},
 {9},
 {10},
 {12},
 }
 $
 Then \begin{align}\label{tildeR4}&{\LRAII^{AII}}^{-1}(V, Q)= \c \circ (\rm{red}_{4}^{-1},\rm{id})\circ (\underset{{\r}}\leftarrow V)\nonumber\\
&=\c\circ (\rm{red}_{t_0}^{-1},\rm{id})((1,5,8,12),V'_{\ge 2})\nonumber\\
&=\c\circ (\rm{red}_{t_0}^{-1}(2,6,7),V'_{\ge 2})\nonumber\\
&=\YT{0.13in}{}{
{1},
 {3},
 {4},
 {5},
 {8},
 {9},
 {10},
 {12},
}\bigcdot\YT{0.13in}{}{
{{1},{1},1},
 {{4},{4},4},
 {5,5},
 {8,8},
 {9},
}
=
\YT{0.13in}{}{
{1,1,{1},1},
 {3,{4},{4},4},
 {4,5,5},
 {5,8,8},
 {8,9},
 {9},
 {10},
 {12},
}:=V^{\r}  \in SST_{12}(\lambda)
\end{align}
 \end{enumerate}

\end{ex}

\end{obs}
\subsection{One or more vertical strips, slack sequences and slack vector sequences}

We now introduce the \emph{slack} \emph{incidence} \emph{matrix} of a recording tableau $Q\in Rec_{2n}(\lambda/\mu)$ which encodes the information as described in  Proposition \ref{prop:slack} below to define ${\LRAII^{AII}}^{-1}$ on $SpT_{2n}\times Rec_{2n}(\lambda/\mu)$.

 \begin{defi}\label{def:tei} Let
 $Q\in  Rec_{2n}(\lambda/\mu)$  defined by the sequence of vertical strips
\begin{align}\label{slack}&\mu^{(0)}=\lambda\supset_{vert} \mu^{(1)}\supset_{vert} \cdots\supset_{vert}\mu^{(N-1)}\supset_{vert}\mu^{(N)} =\mu
\end{align}
such that $2\le |\mu^{(i-1)}/\mu^{(i)}|=Q[i]\in 2\Z$, $1\le i\le N$  and $\nu=( Q[1],\dots,Q[N])^t$ is an even partition. For   $1\le i\le N$,
let $t_0^{(i)}=\ell(\mu^{(i)})-l_0^{(i)}$ where $l_0^{(i)}$ is the number of cells in the vertical strip $\mu^{(i-1)}/\mu^{(i)}$ with row coordinates in $[1,\ell(\mu^{(i)})]$. Then,  for $i=1,\dots,N$,
\begin{enumerate}
\item $t^{(i)}_0\in\{0,1,\dots,\ell(\mu^{(i)}\}$ is the \emph{slack} of the \emph{vertical}  \emph{strip} $\mu^{(i-1)}/\mu^{(i)}$, and
    \item $\r^{(i)}\subseteq [1,\ell(\mu^{(i)})]$ with cardinal $|\r^{(i)}|=t_0^{(i)}$,
   the corresponding \emph{slack} (\emph{row} \emph{index}) \emph{vector}.   We write $\r^{(i)}=()$ or $ \emptyset$ when $t_0^{(i)}=0$.
\end{enumerate}
The sequence of \emph{slack} \emph{numbers} $\underline\t=(t_0^{(N)},\dots,t_0^{(1)})$,  is called  the \emph{slack} \emph{sequence} of $Q$  uniquely defined by \eqref{slack}. Analogously, the corresponding sequence of slack vectors $\underline\r=[\r^{(N)},\dots,\r^{(1)}]$ is called the \emph{slack} (row index) \emph{vector} \emph{sequence} of $Q$.

The $\ell(\mu^{(1)})\times N$ matrix  $[\delta_{\r^{(N)}},\dots,\delta_{\r^{(1)}}]$ is called  the \emph{slack} \emph{incidence} (row) \emph{matrix} of $Q$.
\end{defi}

\begin{ex}In Example \ref{ex:q}, $Q\in Rec_{12}(\lambda/\mu)$, $N=1$, $\mu^{(0)}=\lambda$, $\mu^{(1)}=\mu$ and $t_0^{(1)}=\ell(\mu)-l^{(1)}_0=6-2=4$ with $l_0^{(1)}=2$, and $Q[1]=4$. Thus $\ell(\mu^{(0)})=\ell(\lambda)=Q[1]+t_0^{(1)}=4+4$ and $Q[1]+2t_0^{(1)}\le 12$. The slack row index vector is $\r=(1,3,4,6)$, and the $\ell(\mu)\times 1$ slack incidence  matrix is $[\delta_{\r^{(1)}}]=[1\;0\;1\;1\;0\;1]^T$

\end{ex}

Given $Q\in  Rec_{2n}(\lambda/\mu)$, Proposition \ref{prop:tilderec}, $(R3)$, $(R4)$ impose conditions to the slack sequence  of  $Q$, and  $(R5)$ impose conditions to its  \emph{slack} \emph{vector} \emph{sequence} $\underline\r=[\r^{(N)},\dots,\r^{(1)}]$ equivalently to the slack incidence matrix $[\delta_{\r^{(N)}},\dots,\delta_{\r^{(1)}}]$ of $Q$.

Given $x\in\Z^p_\ge 0$, $y\in\mathbb{Z}^q_{\ge 0}$, we write $x\le_{\r} y$ to mean $p\ge q\ge 0$, and $x_i\le y_i$ whenever $y_i>0$. In this case, $x\le_\r ()$.

\begin{prop}\label{prop:slack} Let
 $Q\in  Rec_{2n}(\lambda/\mu)$ with vertical strip decomposition \eqref{slack}, and put $t^{(0)}_0:=0$ and $\r^{(0)}:=()$. Then, for each $1\le  i\le N$,

\begin{itemize}

\item[(a)]
 $\ell(\mu^{(i-1)})=Q[i]+t_0^{(i)}\ge \ell(\mu^{(i)}$,

\item[(b)] $0\le t^{(i-1)}_0\le t^{(i)}_{0} \le \ell(\mu^{(i)})\le \ell(\mu^{(i-1)})$. The slack sequence $\underline\t=(t^{(N)}_0\ge \dots \ge t^{(1)}_0)$ is weakly decreasing while the sequence $(\ell(\mu^{(N)})\le \dots \le \ell(\mu^{(1)})\le \ell(\mu^{(0)}) )$ is weakly increasing.

\item[(c)] $2n\ge Q[i]+2t_0^{(i)}=\ell(\mu^{(i-1)})+t_0^{(i)}\Leftrightarrow 2n-t_0^{(i)}\ge Q[i]+t_0^{(i)}=\ell(\mu^{(i-1)})$.

\item[(d)] $\r^{(i)}\le_\r \r^{(i-1)}$. The slack vector  sequence $\underline\r=(\r^{(N)}\le_\r\cdots\le_r\r^{(1)})$ is weakly increasing.
    \item[(e)] $\r^{(i)}\subseteq[1,\ell(\mu^{(i)}]\subseteq [1,\ell(\mu^{(i-1)}]\subseteq[1,2n-t_0^{(i)}]$.
\end{itemize}
\end{prop}
\begin{proof} $(\rm d)$ follows from Definition \ref{def:tei} and $(\rm a),(\rm b),(\rm c)$.
\end{proof}

Notably, condition $(\rm d)$ above guarantees that reverse Schensted column insertion can  correctly be $\r$-iterated to define ${\LRAII^{AII}}^{-1}$ as required in Remarks 1 and 2 in \cite{azreco}.

\begin{obs} \label{re:slack} \begin{enumerate}

\item When the slack sequence of $Q$ is the null vector then the conditions on $Q$ are just $\ell(\mu^{(i-1)})=Q[i]\ge \ell(\mu^{(i)})$ and
$2n\ge Q[i]$, for $i=1,\dots, N$. Therefore, $\lambda=\mu+\sum_{i=1}^N\varpi_{Q[i]}$.

\item The admissible slacks for  $n=1,2, 3$ are as follows. Let $N\ge 1$  and $1\le  i\le N$. Proposition \ref{prop:slack} $(b)$ forces $ t_0^{(i)}=0\Rightarrow t_0^{(j)}=0$ for all $i\ge j$.

For $n=1$,   Proposition \ref{prop:slack} $(c)$ implies $2\ge Q[i]+2t_0^{(i)}\ge 2 +2t_0^{(i)}\Leftrightarrow t_0^{(i)}=0$. Hence $Q[i]=2$ and $t_0^{(i)}=0$.

For $n=2$,  Proposition \ref{prop:slack} $(c)$ implies $4\ge Q[i]+2t_0^{(i)}\ge 2 +2t_0^{(i)}\Leftrightarrow 2\ge 2t_0^{(i)}\Leftrightarrow 1\ge t_0^{(i)}$.
Hence either $Q[i]=4$ and $t_0^{(i)}=0$ or $Q[i]=2$ and $t_0^{(i)}=0,1$. Note, $t_0^{(i)}=1\Rightarrow 2\times 2-1\ge Q[i]+1\Rightarrow Q[i]=2$. In particular $4\notin \r^{(i)}$, for any $i$.

For $n=3$,  Proposition \ref{prop:slack} $(c)$ implies $6\ge Q[i]+2t_0^{(i)}\ge 2 +2t_0^{(i)}\Leftrightarrow 2\ge t_0^{(i)} $.
Hence either $Q[i]=6$ and $t_0^{(i)}=0$ or $Q[i]=4$ and $t_0^{(i)}=0,1$ or { $Q[i]=2$} and $t_0^{(i)}=0,1,2$.  In particular $6\notin \r^{(i)}$, for any $i$.
\end{enumerate}
\end{obs}

\begin{ex}\label{ex:vectorslackn=3} Let $n=3$, $S=\YT{0.2in}{}{
 {2},
 {3},
 {6},
} \in SpT_6(\varpi_3)$ and $\underline \t=(2,2,1,1,0)$, $\underline\r=[(2,3);(3,4);4;5;()]$ and $$Q=\YT{0.12in}{}{
 {,5,4,3,2,1},
 {,4,3,2,1},
 {,3,2,1},
 {5,2,1},
 {3,1},
 {1},
} \in Rec_{6}((6,5,4,3,2,1)/(1,1,1))\quad
\rm\delta_{\underline\r}=\begin{bmatrix}
0&0&0&0&0\\
1&0&0&0&0\\
1&1&0&0&0\\
0&1&1&0&0\\
0&0&0&1&0\\
0&0&0&0&0\\
\end{bmatrix}_{6\times 5}\quad (2,3)\ge_{\r}(3,4)\ge_{\r}4\ge_{\r}5\ge_{\r}()$$

\begin{align*}& \c \circ (\rm{red}_{2}^{-1},\rm{id})\circ (\underset{{(2,3)}}\leftarrow \YT{0.12in}{}{
 {2},
 {3},
 {6},
} )\nonumber\\
&=\c\circ (\rm{red}_{2}^{-1},\rm{id})((36), 2)=\YT{0.12in}{}{
 {1,2},
 {2},
 {3},
 {6},
}=S^{(2,3)}\nonumber\\
&\c \circ (\rm{red}_{2}^{-1},\rm{id})\circ (\underset{{(3,4)}}\leftarrow S^{(2,3)})\nonumber\\
&=\c\circ (\rm{red}_{2}^{-1},\rm{id})((36), \YT{0.12in}{}{
 {1,2},
 {2},
})=\YT{0.12in}{}{
 {1,1,2},
 {2,2},
 {3},
 {6},
}=S^{(2,3),(3,4)}\\
\end{align*}
\begin{align*}
& \c \circ (\rm{red}_{2}^{-1},\rm{id})\circ (\underset{{4}}\leftarrow S^{(2,3),(3,4)})\\
&=\c \circ (\rm{red}_{2}^{-1},\rm{id})(6,\YT{0.12in}{}{
 {1,1,2},
 {2,2},
 {3},
 }
 )=\YT{0.12in}{}{
 {1,1,1,2},
 {2,2,2},
 {3,3},
 {4},
 {6},
}=S^{(2,3),(3,4),4}
\end{align*}
\begin{align*}
& \c \circ (\rm{red}_{2}^{-1},\rm{id})\circ (\underset{{5}}\leftarrow S^{(2,3),(3,4),4})\\
&=\c \circ (\rm{red}_{2}^{-1},\rm{id})(6,\YT{0.12in}{}{
 {1,1,1,2},
 {2,2,2},
 {3,3},
 {4},
 }
 )=\YT{0.12in}{}{
 {1,1,1,1,2},
 {2,2,2,2},
 {3,3,3},
 {4,4},
 {6},
}=S^{(2,3),(3,4),4,5}
\end{align*}
\begin{align*}
& \c \circ (\rm{red}_{0}^{-1},\rm{id})\circ (\underset{{()}}\leftarrow S^{(2,3),(3,4),4,5})\\
 &=\YT{0.12in}{}{
 {1,1,1,1,1,2},
 {2,2,2,2,2},
 {3,3,3,3},
 {4,4,4},
 {5,6},
 {6},
}=S^{(2,3),(3,4),4,5,()}\in SST_{6}(6,5,4,3,2,1).
\end{align*}
The output tableau is a $\k$-highest weight tableau \cite{nsw} as explained in Section \ref{sec:kk}.
\end{ex}
The following considers the case where no bumping is needed, that is, the slack sequence is the null vector, $\underline\t=(t^{(N)}_0\ge \dots \ge t^{(1)}_0)=(0,\dots,0)$. When the slack number is $0$ the corresponding slack vector is written $\r=()$. In this case, then \emph{}{slack} \emph{vector} {sequence} is written $\underline\r=((),\dots,())$. It follows from Proposition \ref{prop:slack} and the iteration of Remark \ref{nullslack}.

 \begin{lem}\cite{azreco}  \label{lem:recordhole0} Let $S\in SpT_{2n}(\mu)$ and  let $Q\in  Rec_{2n}(\lambda/\mu)$ such that $\lambda=\mu+\sum_{i=1}^{\nu_1}\varpi_{Q[i]}$
 with $Q[1]\ge\cdots\ge Q[N]\ge \ell(\mu)$  and $\nu=( Q[1],\dots,Q[N])^t$ an even partition.

Then \begin{align}{\LRAII^{AII}}^{-1}(S, Q)=Y(\nu)\bigcdot S \in SST_{2n}(\lambda)
\end{align}
where $Y(\nu)$ is the Yamanouchi tableau of shape $\nu$.
\end{lem}

Next one considers constant slack  sequences   $\underline\t=(t^{(N)}_0\ge \dots \ge t^{(1)}_0)=(1,1,\dots,1,0,\dots,0)$ possibly with a tail of zeroes. Henceforth, the slack vector sequence $\underline\r$ is an increasing sequence of numbers possibly with a corresponding tail of emptysets.
 From Theorem \ref{thm:verygeneralstrip}, \eqref{inverse}, Proposition \ref{prop:slack}, Lemma \ref{lem:recordhole0} and Remark 1 in \cite{azreco}  one has.

\begin{thm}\cite{azreco}\label{lem:recordhole1} Let $S\in SpT_{2n}(\mu)$ and let
 $Q\in   Rec_{2n}(\lambda/\mu)$  defined by the sequence of vertical strips
\begin{align}&\mu^{(0)}=\lambda\supset_{vert} \mu^{(1)}\supset_{vert} \cdots\supset_{vert}\mu^{(N-1)}\supset_{vert}\mu^{(N)} =\mu
\end{align}
such that $|\mu^{(i-1)}/\mu^{(i)}|=Q[i]\ge 2$ and $\ell(\mu^{(i-1)})=Q[i]+1$, $1\le i\le N$  and $( Q[1],\dots,Q[N])^t$ is an even partition.
For $i=1,\dots, N$, let  $ \r_i\in [1,\ell(\mu^{(i)})]$ be  the sole row index of $\mu^{(i)}$ not  a cell row coordinate  of the vertical strip $\mu^{(i-1)}/\mu^{(i)}$, or $\r_i=()$. Then,
\begin{enumerate}
\item for $1\le i\le N$,

\begin{itemize}
\item[(i)] $2n-1\ge Q[i]+1=\ell(\mu^{(i-1)})$, whenever $\r_i\neq ()$. 
 \item [(ii)]$\r_{N}\le\cdots\le \r_1$, $1\le \r_i\le \ell(\mu^{(i)})$, $1\le i\le N$.
and $\r_i\le 2n-1$ whenever $\r_i\neq ()$. (Here, for simplicity, since $\r_i\neq ()$ is a single set we abuse notation and identify $\r_i$ with its single element.)
\end{itemize}
\item
\begin{align}{\LRAII^{AII}}^{-1}(S, Q)=S^{\underline \r}=S^{\r_N\cdots\r_2 \r_1}:=
{({(\cdots{((S^{\r_N})}^{\r_{N-1}}){\cdots})}^{\r_2})}^{\r_1}.
\end{align}
\end{enumerate}

\end{thm}

More generally from Theorem  \ref{thm:verygeneralstrip}, \eqref{inverse},  Proposition \ref{prop:slack}, Theorem \ref{lem:recordhole0}, Lemma \ref{lem:recordhole1} and Remark 1 in \cite{azreco}.

\begin{thm}\label{cor:generalhole1} Let $S\in SpT_{2n}(\mu)$ and let
 $Q\in   Rec_{2n}(\lambda/\mu)$  defined by the sequence of vertical strips
\begin{align}&\mu^{(0)}=\lambda\supset_{vert} \mu^{(1)}\supset_{vert} \cdots\supset_{vert}\mu^{(N-1)}\supset_{vert}\mu^{(N)} =\mu
\end{align}
with slack sequence  $\underline \t=(t^{(N)}_0\ge \dots \ge t^{(1)}_0)$ such that $|\mu^{(i-1)}/\mu^{(i)}|=Q[i]\ge 2$ and $\ell(\mu^{(i-1)})=Q[i]+t_0^{(i)}$, $1\le i\le N$  and $( Q[1],\dots,Q[N])^t$ is an even partition. Let  $\underline\r=[\r^{(N)},\dots,\r^{(1)}]$ be the
{slack} (row index) {vector} {sequence} of $Q$.
Then
\begin{enumerate}
\item for $1\le i\le N$,
\begin{itemize}
\item[(i)] $\r^{(N)}\le_\r\cdots\le_\r \r^{(1)}$ and $\r^{(i)}\subseteq[1,\ell(\mu^{(i)}]\subseteq [1,\ell(\mu^{(i-1)}]\subseteq[1,2n-t_0^{(i)}]$.
\item [(ii)]
$2n-t_0^{(i)}\ge Q[i]+t_0^{(i)}=\ell(\mu^{(i-1)}$.
\end{itemize}
\item
\begin{align}{\LRAII^{AII}}^{-1}(S, Q)=S^{\underline \r}=S^{\r^{(N)}\cdots\r^{(2)} \r^{(1)}}:=
{\bigg({\bigg(\cdots{\big(\big(S^{\r^{(N)}}\big)}^{\r^{(N-1)}}\big)\cdots\bigg)}^{\r^{(2)}}   \bigg)}^{\r^{(1)}}.
\end{align}
\end{enumerate}

\end{thm}

\section{  The  inverse quantum LR map as a way to compute $\mathfrak{k}$-highest and lowest weight tableau pairs }\label{sec:kk}

We now apply the explicit inverse of the quantum Littlewood-Richardson map to  compute  $\mathfrak{k}$-highest or lowest weight tableaux in the
 proof of the Naito–Sagaki conjecture via the branching rule for $\imath$quantum groups \cite{nsw}.
We closely follow the notation in \cite{nsw} and refer to it for details.
 Let $n\in\mathbb{N}$ and let us consider the two sequences of positive integers defined in \cite{nsw}

\begin{align}&u_k=2k-\frac{1+(-1)^k}{2}\label{numbers:u}\\
&v_k=2k-\frac{1+(-1)^{k+1}}{2},~~ k=1,\dots,n.\label{numbers:v}
\end{align}
These two sequences $\{u_i\}_{i=1}^n\subseteq [1,2n]$ and $\{v_i\}_{i=1}^n\subseteq [1,2n]$  have no common values and its union  gives  $\{u_i\}_{i=1}^n\sqcup\{v_i\}_{i=1}^n=[1,2n]$.

\begin{ex} For $n=6$, $u_i=2,3,6,7,10,11$ and $v_i=1,4,5,8,9,12$, $1\le i\le 6$, and $\{u_i\}_{i=1}^6\cup \{v_i\}_{i=1}^6=\{1,2,\dots,12\}$; and for $n=7$, $u_i=2,3,6,7,10,11, 14$ and $v_i=1,4,5,8,9,12,13$, $1\le i\le 7$,
and $\{u_i\}_{i=1}^6\cup \{v_i\}_{i=1}^6=\{1,2,\dots,14\}$.
\end{ex}

For $S\in SST_{2n}(\lambda)$, its $\mathfrak{k}$-weight is \cite[Section 5.3]{nsw}
\begin{align}\wt_\mathfrak{k}(S)=(S[u_1]-S[v_1])\tilde\varepsilon_1+(S[u_2]-S[v_2])\tilde\varepsilon_2+\dots+(S[u_n]-S[v_n])\tilde\varepsilon_n.\label{kweight}
\end{align}


The set $\widetilde P^+=\{\mu_1 \tilde\varepsilon_1+\cdots+\mu_n \tilde\varepsilon_n\in\widetilde P:\mu_1\ge\cdots\ge \mu_n\ge 0\}$ can be identified with the set $Par_{\le n}$ \cite[Section 5.1]{nsw}

If $S$ is  the symplectic tableau below on the LHS \eqref{symphw-lw}  then the shape is equal to  the  corresponding $\mathfrak{k}$-weight,
  $wt_\mathfrak{k}(S)=(S[u_1],S[u_2],\dots,S[u_n])$,  and if $S$ is on the RHS \eqref{symphw-lw} then the $\mathfrak{k}$-weight is $(-S[v_1],-S[v_2],\dots,-S[v_n])$  and the shape is
  $$w_0wt_\mathfrak{k}(S)=w_0(-S[v_1],-S[v_2],\dots,-S[v_n])=(S[v_1],S[v_2],\dots,S[v_n])$$
  \noindent where $w_0$ is the longest element of the Weyl group of the Lie algebra $\mathfrak{k}$ isomorphic to the symplectic Lie algebra $\mathfrak{sp}_{2n}(\mathbb{C})$,

\begin{align}\label{symphw-lw}
&\YT{0.2in}{}{
 {{u_1},{u_1},\cdots,\cdots,\cdots,\cdots,u_1},
 {{u_2},\cdots,\cdots,\cdots,\cdots,{u_2}},
 {{$\vdots$},\vdots,\vdots,{$\vdots$}},
 {u_n,\cdots,u_n},
} \qquad
\YT{0.2in}{}{
 {{v_1},{v_1},\cdots,\cdots,\cdots,v_1},
 {{v_2},\cdots,\cdots,\cdots,{v_2}},
 {{$\vdots$},\vdots,\vdots,{$\vdots$}},
 {v_n,\cdots,v_n},
}  \\
&\mbox{ \scriptsize symplectic $\mathfrak{k}$-highest weight tableau}\qquad \mbox{ \scriptsize symplectic $\mathfrak{k}$-lowest weight tableau}
\nonumber
\end{align}

Let $SST_{2n}$ denote the set of all semi-standard tableaux in  the alphabet $[1,2n]$ and $SpT_{2n}$ its subset  of symplectic tableaux.
Next, $\mathfrak{k}$-highest weight tableaux and $\mathfrak{k}$-lowest weight tableaux in $SST_{2n}$ are defined.
\begin{defi} \cite{nsw} Let $S\in SST_{2n}$.
\begin{enumerate}
\item
$S$ is called a $\mathfrak{k}$-highest weight tableau if $P^{AII}(S)$ is a symplectic tableau of the form shown in the LHS of \eqref{symphw-lw}; let $SST^{\mathfrak{k}-hw}
_{2n}(\lambda)\subseteq SST_{2n}(\lambda)$ denote the set of all $\mathfrak{k}$-highest weight tableaux of shape $\lambda$.

\item
$S$ is called a $\mathfrak{k}$-lowest weight tableau if $P^{AII}(S)$ is a symplectic tableau of the form shown in RHS of \eqref{symphw-lw}; let $SST^{\mathfrak{k}-lw}
_{2n}(\lambda)\subseteq SST_{2n}(\lambda)$ denote the set of all $\mathfrak{k}$-lowest weight tableaux of shape $\lambda$.
\end{enumerate}
\end{defi}

Since $P^{AII}(S)=S$ for $S\in SpT_{2n}$, this definition for $S\in SpT_{2n}$ obviously restricts to \eqref{symphw-lw}.
For $n=1,2,3$, the symplectic $\mathfrak{k}$-highest weight respectively $\mathfrak{k}$-lowest weight tableaux are then respectively of the form
\begin{align}\label{1symphw-lw2}n=1:~~
&S^H=\YT{0.2in}{}{
 {{2},{2},\cdots,\cdots,\cdots,\cdots,2},
},\qquad S_L=\YT{0.2in}{}{
 {{1},{1},\cdots,\cdots,\cdots,1},
 }\in SpT_2;\\\nonumber
 \\
n=2:~~&S^H=\YT{0.2in}{}{
 {{2},{2},\cdots,\cdots,\cdots,\cdots,2},
 {{3},\cdots,\cdots,\cdots,\cdots,{3}},
},\qquad S_L= \YT{0.2in}{}{
 {{1},{1},\cdots,\cdots,\cdots,1},
 {{4},\cdots,\cdots,\cdots,{4}},
 }\in SpT_4\label{2symphw-lw2}\\ \nonumber
 \\
 n=3:~~&S^H=\YT{0.2in}{}{
 {{2},{2},2,\cdots,\cdots,\cdots,\cdots,2},
 {{3},3,\cdots,\cdots,\cdots,\cdots,{3}},
 {{6},\cdots,\cdots,\cdots,\cdots,{6}},
},\qquad S_L= \YT{0.2in}{}{
 {{1},{1},1,\cdots,\cdots,\cdots,1},
 {{4},4,\cdots,\cdots,\cdots,{4}},
 {{5},\cdots,\cdots,\cdots,{5}},
 }\in SpT_6\label{3symphw-lw2}
\end{align}


It would be interesting to have an explicit characterization  of $SST^{\mathfrak{k}-hw}
_{2n}(\lambda)$ or $SST^{\mathfrak{k}-lw}
_{2n}(\lambda)$. This is known for $n=1, 2$ \cite{nsw}. One has an algorithm provided by our results to compute the elements of those sets.

From the quantum Littlewood-Richardson bijection \eqref{wat0}, we see that for each $\mu\in Par_{\le n}$, the recording tableaux $Rec_{2n}(\lambda/\mu)\overset{\sim}\rightarrow LRS_{2n}(\lambda/\mu)$ determine the pairs consisting of $\mathfrak{k}$-highest  respectively $\mathfrak{k}$-lowest weight tableaux $S^{H,\mu}$ and $ S_{L,-\mu} \in
SST_{2n}(\lambda)$ where $wt_{\mathfrak{k}}(S^{H,\mu})=\mu$  respectively $wt_{\mathfrak{k}}(S_{L,-\mu})=-\mu$. That is, the set $Rec_{2n}(\lambda/\mu)$ determine the tableaux in
$SST^{\mathfrak{k}-hw}
_{2n}(\lambda)$  respectively $SST^{\mathfrak{k}-lw}
_{2n}(\lambda)$) such that $wt_{\mathfrak{k}}(T)=\mu$  and respectively $wt_{\mathfrak{k}}(T)=-\mu$. Consequently, our  algorithm    showing explicitly the surjectivity of $LR^{AII}$ in \eqref{Rtilde} provides  an algorithm for ${\LRAII^{AII}}^{-1}$ to compute those sets.

\begin{prop} \label{prop:hwlw} Let $S^{H,\mu}, S_{L,-\mu}\in SpT_{2n}(\mu)$ be the symplectic $\mathfrak{k}-hw$ respectively $\mathfrak{k}-lw$ weight tableaux of $SpT_{2n}(\mu)$. Then
\begin{enumerate}
\item
 ${\LRAII^{AII}}^{-1}(\{S^{H,\mu} \}\times Rec_{2n}(\lambda/\mu))= \{S\in    SST^{\mathfrak{k}-hw}
_{2n}(\lambda)| wt_{\mathfrak{k}}(S)=\mu\} \subseteq SST_{2n}(\lambda)$.

\item ${\LRAII^{AII}}^{-1}(\{S_{L,-\mu} \}\times Rec_{2n}(\lambda/\mu))=    \{S\in    SST^{\mathfrak{k}-lw}
_{2n}(\lambda)| wt_{\mathfrak{k}}(S)=-\mu\} \subseteq SST_{2n}(\lambda)$.

\end{enumerate}
Therefore \begin{align}SST^{\mathfrak{k}-hw}
_{2n}(\lambda)=\bigsqcup_{\begin{smallmatrix}\mu\in Par_{\le n}\\
\mu\subseteq\lambda\end{smallmatrix}} {\LRAII^{AII}}^{-1}(\{S^{H,\mu}\}\times Rec_{2n}(\lambda/\mu))
\end{align}
and
\begin{align}SST^{\mathfrak{k}-lw}
_{2n}(\lambda)=\bigsqcup_{\begin{smallmatrix}\mu\in Par_{\le n}\\
\mu\subseteq\lambda\end{smallmatrix}} {\LRAII^{AII}}^{-1}(\{S_{L,-\mu}\}\times Rec_{2n}(\lambda/\mu))
\end{align}
\end{prop}

Let $\mu\subseteq_{vert}\lambda$, $t_0$ the slack of the vertical strip $\lambda/\mu$ and $\r$ the corresponding  slack row index vector. Let $\mu'=\mu-\delta_\mathbf{r}$ and note $\lambda=\mu'+\varpi_{\ell(\lambda)}$.
For  $S$  the $\mathfrak{k}$-highest or lowest weight tableau
in $SpT_{2n}(\mu)$,  Theorem \ref{thm:verygeneralstrip} has a nice assertion.

\begin{cor}\label{lem:H-L}Let $S^{H,\mu}$ be the $\mathfrak{k}$-highest weight tableau and $S_{L,-\mu}$ be  $\mathfrak{k}$-lowest weight tableau $S_{L,-\mu}$ 
in $SpT_{2n}(\mu)$ as in  \eqref{2symphw-lw2}.
Let $Q\in Rec_{2n}(\lambda/\mu)$ with slack row index  vector $ \mathbf{r}=(r_1,\dots, r_{t_0})$. Let $u_\r=(u_{r_1},\dots, u_{r_{t_0}})$ and $v_\r=(v_{r_1},\dots, v_{r_{t_0}})$ sequence of positive numbers in [1,2n] as in \eqref{numbers:u}.
Then,
\begin{enumerate}
\item
  ${LR^{AII}}^{-1}(S^{H,\mu},Q)$ returns the following
  $\mathfrak{k}$-highest  weight tableau in
$SST_{2n}(\lambda)$ with ${\mathfrak{k}}$-weight  $\mu$:

\begin{align}\label{TUH}(T^u_0(u_{r_1})T^u_1\cdots (u_{r_{t_0}})T^u_{t_0})S^{H,\mu'}\in SST_{2n}(\lambda)
\end{align}  where  $S^{H,\mu'}$ is the  $\mathfrak{k}$-highest in  $SpT_{2n}(\mu')$
and $T^u_0(u_{r_1})T^u_1\cdots (u_{r_{t_0}})T^u_{t_0}\in SST_{2n}(\varpi_{\ell(\lambda)})$ is  given by Theorem 2 in \cite{azreduction}
\item  ${LR^{AII}}^{-1}(S_{L,-\mu},Q)$ returns the
  $\mathfrak{k}$-lowest  weight tableau in
$SST_{2n}(\lambda)$ with ${\mathfrak{k}}$-weight  $-\mu$:

\begin{align}\label{TVL}(T^v_0(v_{r_1})T^v_1\cdots(v_{r_{t_0}}) T^v_{t_0})S_{L,\mu'}\in SST_{2n}(\lambda)
\end{align}  where  $S_{L,-\mu'}$ is the  $\mathfrak{k}$-lowest weight tableau in  $SpT_{2n}(\mu')$
and $T^v_0(v_{r_1})T^v_1\cdots(v_{r_{t_0}})T^v_{t_0}\in SST_{2n}(\varpi_{\ell(\lambda)})$, given  by Theorem 2 in \cite{azreduction}
\end{enumerate}

\end{cor}

\begin{proof} First note that the sequences of numbers in \eqref{numbers:u} and \eqref{numbers:v} define symplectic columns in the conditions of \cite[Lemma 3]{azreduction}  and therefore  \cite[Theorem2]{azreduction} applies to $u_\r =
(u_{r_1}, \dots , u_{r_{t_0}} )$ and $v_\r = (v_{r_1}, \dots, v_{\r_{t_0 }})$.  Note $\mu'=\mu-\delta_\mathbf{r}\Leftrightarrow \mu=\mu'+\delta_\mathbf{r},$  and $S'_\r=u_\r\in SpT_{2n}(\varpi_{t_0})$ and $\lambda=\mu'+\varpi_l$. Then following the rules in Theorem \ref{thm:verygeneralstrip}, one has

\begin{align}\label{tildeR5}&{LR^{AII}}^{-1}(S^{H,\mu}, Q)= c \circ (\rm{red}_{t_0}^{-1},\rm{id})\circ (\underset{{\mu/\mu'}}\leftarrow S^{H,\mu})\nonumber\\
&=c\circ (\rm{red}_{t_0}^{-1},\rm{id})(u_\r,S^{H,\mu'})\nonumber\\
&=c\circ (\rm{red}_{t_0}^{-1}(u_\r),S^{H,\mu'})\nonumber\\
&=c\circ (T^u_0(u_{r_1})T^u_1\cdots (u_{r_{t_0}})T^u_{t_0},S^{H,\mu'}) \mbox{ by \cite[Theorem2]{azreduction}}\\
&=(T^u_0(u_{r_1})T^u_1\cdots (u_{r_{t_0}})T^u_{t_0})S^{H,\mu'}\in SST_{2n}(\lambda)
\end{align}

The proof for $S_{L,-\mu}$ is similar.
\end{proof}

This result is illustrated for $n=6$ in Example \ref{ex:qq}.
\begin{obs}
Given  $\mu\subseteq_{vert}\lambda$, the data given by the slack number $t_0$ and the slack row index vector $\r$ of $\lambda/\mu$ completely characterizes the pair of $\mathfrak{k}$-highest, -lowest weight tableaux in $SST_{2n}(\lambda)$ with $\mathfrak{k}$- weights $\mu$, -$\mu$).
That is, for $\mu\subseteq_{vert}\lambda$,  $S\in    SST^{\mathfrak{k}-hw}_{2n}(\lambda)$ ($SST^{\mathfrak{k}-lw}_{2n}(\lambda)$)  with $wt_{\mathfrak{k}}(S)=\mu$ ($wt_{\mathfrak{k}}(S)=-\mu$) if and only if $S=$\eqref{TUH} ($=$\eqref{TVL}).

The procedure is very easy: given $S^{H,\mu'}\in SpT_{2n}(\mu')$ and $0\le k\le n $ choose a subset $u=(u_{r_1},\dots, u_{r_k})\subseteq \{u_i\}_{k=1}^n$ \eqref{numbers:u} such that $\mu=\mu'+\delta_\r\in Par_{\le n}$ with $\r=({r_1},\dots, {r_k})\subseteq [1,n]$. Then $u=(u_{r_1},\dots, u_{r_k})\in SpT_{2n}(\varpi_k)$  and
\begin{align}
\redu_{k}^{-1}(u).S^{H,\mu'}\in SST^{\mathfrak{k}-hw}_{2n}(\lambda) \mbox{ with $\mathfrak{k}$-weight $\mu$ }
\end{align}

such that  $\lambda=\mu'+\varpi_l$ where $k\le l\le 2n, ~2n-l$ and  $l-k\in 2\Z$.
\end{obs}
\begin{ex} Let $n=3$ and $S'=\YT{0.13in}{}{
 {2,2,2,2},
 {3,3},
}\in SpT_6(\mu')$ and $u=\{6\}\subseteq \{u_1=2,u_2=3,u_3=6\}$, $\r=\{3\}$. One has $\mu=\mu'+\delta_\r=(4,2,1)=(3,3,1,1)$.
Then $\redu_1^{-1} (6)=(12346)$ and $$S=(12346)\bigcdot S'=\YT{0.13in}{}{
 {1,2,2,2,2},
 {2,3,3},
{3},
{4},
{6},
}\in SST^{\mathfrak{k}-hw}_{2\times 3}(\lambda), ~~l=5, ~~wt_{\mathfrak{k}}(S)=\mu$$
where $\lambda=\mu'+\varpi_5$ and $1=k\le l\le 6, ~6-5$ and  $l-k=5-1=4\in 2\Z$
\end{ex}

From Lemma \ref{lem:recordhole0} it follows
\begin{cor} \label{cor:n=1}Let $S_1$ a $\mathfrak{k}$-highest weight tableau  and $S_2$ a $\mathfrak{k}$-lowest weight tableau in $SpT_2(u)$, with $(u)$ a one-row partition, as in  \eqref{1symphw-lw2}. Let $Q\in  Rec_{2}(\lambda/(u))$ of weight $\nu^t$.
Then for $n=1$, the $\mathfrak{k}$-highest (-lowest) weight tableaux in
$SST_{2}(\lambda)$ are respectively the hook-shape tableaux
\begin{align}&Y(\nu).S_1,\quad \YT{0.13in}{}{
 {1,\cdots,1,{2},{2},\cdots,\cdots,\cdots,2},
{2,\cdots,2},
}\mbox{ and }\\
& Y(\nu).S_2,\quad \YT{0.13in}{}{
 {1,\cdots,1,{1},{1},\cdots,\cdots,\cdots,1},
{2,\cdots,2},
}
\end{align}
with ${\mathfrak{k}}$-weights respectively $(u)$ and $(-u)$.
\end{cor}

From Proposition \ref{prop:slack}, Remark \ref{re:slack} and   Lemma \ref{lem:recordhole1} it follows

\begin{cor}\label{cor:n=2}Let $n=2$. Consider $S^H$  the $\mathfrak{k}$-highest weight tableau
in $SpT_4(\mu)$ as in LHS of \eqref{2symphw-lw2}. Let $Q\in  Rec_{4}(\lambda/\mu)$.
Then, 
\begin{enumerate}
\item the slack sequence and the corresponding  slack row index vector sequence of $Q$ are of the form  $(1,\dots,1,0,\dots,0)$, respectively
 $(1,\cdots,1,2,\cdots,2,3,\cdots,3,(),\dots,())$, and
  \item ${\LRAII^{AII}}^{-1}(S^H,Q)$ returns   a $\mathfrak{k}$-highest  weight tableau in
$SST_{4}(\lambda)$ with ${\mathfrak{k}}$-weight  $\mu$    as   described in \cite[Lemma 6.2,(c)]{nsw}:
\begin{align}\label{type1}
&\YT{0.15in}{}{
 {1,\cdots,1,1,\cdots,1,1,\cdots,1,{1},\cdots,1,2,\cdots,2,\mathbf{2},\cdots,\mathbf{2}},
 {2,\cdots,2,2,\cdots,2,2,\cdots,2,{2},\cdots,2,{\mathbf{3}},\cdots,\mathbf{3}},
 {3,\cdots,3,\mathbf{3},\cdots,\mathbf{3},\mathbf{4},\cdots,\mathbf{4}},
 {4,\cdots,4},
} \quad 
\\
&\mbox{ or }\nonumber\\
\label{type2}
&\YT{0.15in}{}{
 {1,\cdots,1,1,\cdots,1,1,\cdots,1,{2},\cdots,2,2,\cdots,2,\mathbf{2},\cdots,\mathbf{2}},
 {2,\cdots,2,2,\cdots,2,2,\cdots,2,{3},\cdots,{3},\mathbf{3},\cdots,\mathbf{3}},
 {3,\cdots,3,\mathbf{3},\cdots,\mathbf{3},\mathbf{4},\cdots,\mathbf{4},{4},\cdots,{4}},
 {4,\cdots,4},
}
\end{align}
These tableaux satisfy inequalities  on the multiplicities 
of the following columns: if $x$ is the multiplicity of $\redu^{-1}(3)=(1,2,3)$, $y$ the multiplicity of $\redu^{-1}(4)=(1,2,4)$, $z$ the multiplicity of $(2,3)\in SpT_4(\varpi_2)$ and w the multiplicity of $(2)\in SpT_4(\varpi_1) $, then

$$x\le w \mbox{ and } y\le z.$$
Tableaux of type \eqref{type1} are produced with slack vectors where $\#3$'s$>\#1$'$\ge 0$, and those of type \eqref{type2} with slack vectors where $0\le\#3$'s$\le\#1$'.
\end{enumerate}
\end{cor}

\begin{ex}\label{n=2}Let $n=2$. We illustrate Corollary \ref{cor:n=2}.
One has
$\redu_1^{-1}(2)=(234)$, $\redu_1^{-1}(3)=(123)$, $\redu_1^{-1}(4)=(124)$, $\redu_1^{-1}(())=(1,2,3,4)$. our tableaux below comprise these columns and the symplectic columns $(2)$, $(2,3)$ and the column $(1,2)$ obtained either with the operation

$$(\underset{3}\leftarrow (1,2,4))=(4,(1,2))=(1,2,4).(1,2) \mbox{ as in \eqref{movetype}}$$  or $$(\underset{3}\leftarrow (1,2,3))=(3,(1,2))=(1,2,3).(1,2)\mbox{ as in \eqref{small}}$$

\begin{enumerate}
\item Let $\mu=(1,1)\subseteq \lambda=(3,2,1)$,  $S=\YT{0.13in}{}{
 {{2}},
 {{3}},
}\in SpT_4(\mu)$, $\mathfrak{k}$-highest weight, and $~Q=\YT{0.13in}{}{
 {{},2,1},
 {{},1},
{2},
}
\in Rec_4(\lambda/\mu)$ with slack vector sequence $(2,3)$

\begin{align}\nonumber
&\c\circ(\redu^{-1},id)\circ(\underset{2}\leftarrow \YT{0.13in}{}{
 {{2}},
 {{3}},
})
=\YT{0.13in}{}{
 {1,{2}},
 {{2}},
 {3}
}=S^2\\
&\c\circ(\redu^{-1},id)\circ(\underset{3}\leftarrow \YT{0.13in}{}{
 {1,{2}},
 {{2}},
 {3}
})= \YT{0.13in}{}{
 {1,1,{2}},
 {2,{2}},
 {3}
}=S^{2,3}\label{small} \mbox{ of type \eqref{type1}}
\end{align}

\item We expand  the explanation of this example, considered in \cite{azreco}, with slack data. Let $\mu=(4,2)\subseteq \lambda=(10,8,5,1)$ and $S=\YT{0.13in}{}{
 {{2},{2},2,2},
 {{3},{3}},
}\in SpT_4(\mu)$, $\mathfrak{k}$-highest weight, and $$~Q=\YT{0.13in}{}{
 {{},{},,,6,5,4,3,2,1},
 {{},{},8,7,4,3,2,1},
{8,7,6,5,1},
{1},
}\in Rec_4(\lambda/\mu).$$ The  slack vector sequence of $Q$ is $\underline\r=(1,1,2,2,3,3,3,\emptyset)$.

\begin{align*}&\c\circ(\redu^{-1},id)\circ(\underset{1}\leftarrow S)=\YT{0.13in}{}{
 {2,{2},2,2},
 {3,{3},{3}},
 {4},
}=S^1,\quad \\
&\c\circ(\redu^{-1},id)\circ(\underset{1}\leftarrow S^1)=\YT{0.13in}{}{
 {2,{2},2,2},
 {3,3,{3},{3}},
 {4,4},
}=S^{1,1}
\end{align*}
\begin{align*}
 &c\circ(\redu^{-1},id)\circ(\underset{2}\leftarrow S^{1,1})=\YT{0.13in}{}{
 {1,2,{2},2,2},
 {2,3,3,{3}},
 {3,4,4},
}=S^{1,1,2},\quad
\end{align*}
\begin{align*}
&\c\circ(\redu^{-1},id)\circ(\underset{2}\leftarrow S^{1,1,2})=\YT{0.12in}{}{
 {1,1,2,{2},2,2},
 {2,2,3,3},
 {3,3,4,4},
}=S^{1,1,2,2}
\end{align*}
\begin{align*}
&\c\circ(\redu^{-1},id)\circ(\underset{3}\leftarrow S^{1,1,2,2})=\YT{0.13in}{}{
 {1,1,1,2,{2},2,2},
 {2,2,2,3,3},
 {3,3,4,4},
}=S^{1,1,2,2,3}
\end{align*}
\begin{align*}
&\c\circ(\redu^{-1},id)\circ(\underset{3}\leftarrow S^{1,1,2,2,3})=\YT{0.13in}{}{
 {1,1,1,1,2,{2},2,2},
 {2,2,2,2,3,3},
 {3,3,4,4},
}=S^{1,1,2,2,3,3} \label{movetype0}
\end{align*}\mbox{up to here we have been in type \eqref{type2}}\\
\begin{align}
&\c\circ(\redu^{-1},id)\circ(\underset{3}\leftarrow S^{1,1,2,2,3,3})=\YT{0.13in}{}{
 {1,1,1,1,1,2,{2},2,2},
 {2,2,2,2,2,3,3},
 {3,3,4,4},
}=S^{1,1,2,2,3,3,3}\\\label{movetype}
&\mbox{ we have moved to  type \eqref{type1}}
\end{align}
\begin{align*}&\c\circ(\redu^{-1},id)\circ(\underset{()}\leftarrow S^{1,1,2,2,3,3,3})=\YT{0.12in}{}{
 {1,1,1,1,1,1,2,{2},2,2},
 {2,2,2,2,2,2,3,3},
 {3,3,3,4,4},
 {4},
}=S^{1,1,2,2,3,3,3,()}\in SST_4(\lambda) \\
& \mbox{ is a $\mathfrak{k}$- highest weight tableau in $SST_4(\lambda)$ with $\mathfrak{k}$-weight $\mu$ of type \eqref{type1}}
\end{align*}
\item  We expand  the explanation of this example, considered in \cite{azreco}, with slack data. Let $\mu=(5,2)$
and $S=\YT{0.13in}{}{
 {{2},{2},2,2,2},
 {{3},{3}},
}\in SpT_4(\mu)$ and
The  slack vector sequence   $\underline\r=(1,1,1,2,2,3,3,())$. Then
$$S^{1,1,1,2,2,3,3,()}=\YT{0.12in}{}{
 {1,1,1,1,1,2,2,{2},2,2},
 {2,2,2,2,2,3,3,3},
 {3,3,3,4,4,4},
 {4},
}$$
is a $\mathfrak{k}$- highest weight tableau in $SST_4(\lambda)$ with $\mathfrak{k}$-weight $\mu$ of type \eqref{type2}.
\end{enumerate}
Expanding the slack vector sequence $(1,1,1,2,2,3,3,())$  with one more $3$ to have the same multiplicity as $1$'s: $(1,1,1,2,2,3,3,3,())$, we still fall on the $\mathfrak{k}$-highest weight tableau  \eqref{type2}:

$$S^{1,1,1,2,2,3,3,3,()}=\YT{0.12in}{}{
 {1,1,1,1,1,1,2,2,{2},2,2},
 {2,2,2,2,2,2,3,3,3},
 {3,3,3,4,4,4},
 {4},
}$$

Expanding with one more $3$ the slack vector has more $3$'s than $1$'s, we change of type  to  fall in type \eqref{type1}. After this the type stabilizes in type \eqref{type1}
$$S^{1,1,1,2,2,3,3,3,3,()}=\YT{0.12in}{}{
 {1,1,1,1,1,1,1,2,2,{2},2,2},
 {2,2,2,2,2,2,2,3,3,3},
 {3,3,3,4,4,4},
 {4},
}$$

\end{ex}

\begin{ex}\label{ex:n=3} For $n=3$, one has
$\redu_1^{-1}(2)=(23456)$, $\redu_1^{-1}(3)=(12356)$, $\redu_1^{-1}(6)=(12346)$, $\redu_1^{-1}(())=(1,2,3,4,5,6)$. Our $\k$-highest weight tableaux  below comprise these columns and the symplectic columns $(2)$, $(2,3)$, $2,3,6$ and the columns  obtained either with the operations

$$(\underset{5}\leftarrow (2,3,4,5,6))=(6,(1,2,3,4,5))=(1,2,3,4,6).(2,3,4,5)$$ 
$$(\underset{5}\leftarrow (1,2,3,5,6))=(6,(1,2,3,5))=(1,2,3,4,6).(1,2,3,5)$$ 
$$(\underset{5}\leftarrow (1,2,3,4,6))=(6,(1,2,3,4))=(1,2,3,4,6).(1,2,3,4)$$

These are the columns that appear in the patterns below

\begin{enumerate}
\item We expand  slack vector the following illustration in \cite{azreco} revealing other patterns as expected. The slack sequence is constant equal to $1$ and $()$.

 Let $n=3$.
Let $\mu=(6,4,2)\subseteq \lambda=(10,9,7,6,5,1)$ and $S=\YT{0.13in}{}{
 {{2},{2},2,2,2,2},
 {{3},{3},3,3},
 {6,6},
}\in SpT_6(\mu)$ and
 the  slack row index vector sequence  $\underline\r=(1,1,2,3,5,())$.

One has
$\redu_1^{-1}(2)=(23456)$, $\redu_1^{-1}(3)=(12356)$, $\redu_1^{-1}(6)=(12346)$, $\redu_1^{-1}(())=(1,2,3,4,5,6)$.
\begin{align*}=S^{1,1,2,3,5,()}=\YT{0.12in}{}{
 {1,1,1,1,2,2,{2},2,2,2},
 {2,2,2,2,3,3,3,3,3},
 {3,3,3,3,4,4,6},
 {4,4,4,5,5,5},
{5,6,6,6,6},
{6},
} 
\end{align*}
\mbox{ is a $\mathfrak{k}$- highest weight tableau in $SST_4(\lambda)$ of $\mathfrak{k}$-weight $\mu=(6,4,2)$}

Expanding the slack row index vector sequence  with two and three more $5$'s, so that $\underline\r=(1,1,2,3,5,5,())$ we get  new patterns
\begin{align*}S^{1,1,2,3,5,5,5,()}=\YT{0.12in}{}{
 {1,1,1,1,1,2,2,{2},2,2,2},
 {2,2,2,2,2,3,3,3,3,3},
 {3,3,3,3,3,4,4,6},
 {4,4,4,4,5,5,5},
{5,6,6,6},
{6},
} 
\end{align*}
and  $\underline\r=(1,1,2,3,5,5,5,())$
\begin{align*}S^{1,1,2,3,5,5,5,()}=\YT{0.12in}{}{
 {1,1,1,1,1,1,2,2,{2},2,2,2},
 {2,2,2,2,2,2,3,3,3,3,3},
 {3,3,3,3,3,3,4,4,6},
 {4,4,4,4,4,5,5,5},
{5,6,6,6},
{6},
} 
\end{align*}
After this the type stabilizes in this form. We draw the attention of the reader for the three first rows of this pattern, entry $6$ apart, which correspond to \eqref{type2}.

\item Another admissible pattern are the ones produced by slack sequences with entries equal to two and tails of $1$'s and $()$ as in Example \ref{ex:vectorslackn=3}.

\end{enumerate}
\end{ex}

The exploration of the case $n=3$ reveals interesting properties of the generation of $\k$-highest weights that should be further studied. The dual information of a recording tableau packed on the slack information is ana attempt to get insight.

\bibliography{sample17}
\bibliographystyle{alpha}

\end{document}